\documentclass[12pt]{amsart}

\usepackage{amscd}
\usepackage{amsmath,graphicx}
\usepackage{verbatim}

\usepackage{amssymb}
\usepackage{pb-diagram,pb-xy}
\usepackage{ bbold }

\usepackage{amssymb}

\usepackage{mathtools}

\usepackage{xy}
\xyoption{all}

\numberwithin{equation}{section}

\DeclareMathAlphabet{\cat}{OT1}{cmss}{m}{sl}

\newtheorem{theorem}[equation]{Theorem}
\newtheorem{proposition}[equation]{Proposition}
\newtheorem{lemma}[equation]{Lemma}
\newtheorem{corollary}[equation]{Corollary}

\theoremstyle{definition}

\newcommand{\tens}{\otimes}
\newcommand{\inv}{^{-1}}

\newcommand{\Gal}{\operatorname{Gal}}

\newcommand{\GL}{\operatorname{GL}}

\newcommand{\xra}{\xrightarrow}

\newcommand{\F}{\mathbb{F}}

\newcommand{\Z}{\mathbb{Z}}

\newcommand{\Q}{\mathbb{Q}}

\renewcommand{\leq}{\leqslant}
\renewcommand{\geq}{\geqslant}

\newcommand{\nichego}[1]{}

\title
[Lifting property for finite groups] % colontitle
{Lifting property for finite groups}

\author{Chandrashekhar B. Khare}
\email{shekhar84112@gmail.com}

\author{Alexander Merkurjev}
\email{merkurev@math.ucla.edu}

\address{Department of Mathematics, University of California, Los Angeles, CA 90095-1555, USA}

\begin{document}

\begin{abstract}
We classify all finite groups that have lifting property of mod $p$ representations to mod $p^2$ representations
for all prime $p$.
\end{abstract}

\maketitle

\section{Introduction}

Let $R$ be a commutative local ring with maximal ideal $M$ satisfying $M^2=0$ and residue field $k$ of positive characteristic $p$.
For example, one can take $R=\Z/p^2\Z$ with $k=\F_p$, or more generally, for a field $k$ of characteristic $p>0$,
$R=W_2(k)$ the ring of the $p$-typical length $2$ Witt vectors of $k$.

The group
homomorphism $\GL_n(R)\to \GL_n(k)$ is surjective for every $n$ and its kernel $A_n(R)$ is an abelian group of exponent $p$.

Let $G$ be a   finite  group.
The following two statements are equivalent:

\medskip

1. Every    group homomorphism $G\to\GL_n(k)$  for any $n>0$ extends to a  group homomorphism $G\to\GL_n(R)$. 

\smallskip

2. Every continuous   $k[G]$-module $M$ of finite $k$-dimension  lifts to a  $R[G]$-module $N$, that is $N$ is free as
an $R$-module,  and $M\simeq N\tens_R k$ as $k[G]$-modules.

\medskip

If these two conditions hold for all $R$ (respectively for $R=\Z/p^2\Z$, where  $p$ is a prime integer) we say that $G$ is \emph{liftable}
(respectively, \emph{$p$-liftable}). The theorem below    shows it  to be a very restrictive property for finite groups.

\begin{theorem}\label{main}
Let $G$ be a finite group. The following are equivalent:

\noindent $(1)$\ $G$ is liftable.

\noindent $(2)$\  $G$ is $p$-liftable for every prime $p$.

\noindent $(3)$\  $G$ is isomorphic to one of the following groups: $C_{2^n}$, $C_3 \times C_{2^n}$ or $C_3 \rtimes  C_{2^n}$.
(The semidirect product is taken with respect the only nontrivial action of $ C_{2^n}$ on $C_3$.)
\end{theorem}

We say a few words about  a context for  our result, and earlier related work. The modularity  conjecture of Serre \cite{Serre}, proved in \cite{KW},  asserts  that odd, irreducible, 2-dimensional  mod $p$  representations $\rho$  of  $\Gal(\overline{\Q}/\Q)$,  \[\rho:\Gal(\overline{\Q}/\Q) \to \GL_2(k),\] with $k$ a finite field,  arise from newforms, and in particular lift to characteristic 0. This has led to a lot of work about lifting  mod $p$ representations of absolute Galois groups of local and global fields.  For absolute Galois groups of  fields,  lifting questions have been studied in \cite{MS26}.  The lifting property for   finite groups was first studied in \cite{KL}: this  is almost complementary to the case  of absolute Galois groups as,   by  a   theorem of Artin and Schreier, the only finite groups that are absolute Galois groups are 1 and $C_2$.

More precisely, in  \cite{KL},  $G$ is  said to be \emph{liftable} (respectively, \emph{$p$-liftable}), if any homomorphism $f:G \rightarrow \GL_n(k)$   with $k$ a finite field (respectively, a finite extension of $\F_p$) lifts to $\GL_n(W_2(k))$.   The definition of liftable in \cite{KL}  is {\it a priori}  weaker than our definition, and of $p$-liftable in \cite{KL}  is {\it a priori}  stronger than our definition. Under  either  of   the definitions,  it is easy to see that a  finite group $G$  is $p$-liftable if and only if a  Sylow $p$-subgroup of $G$  is  $p$-liftable. From this and the equivalence of (1) and (2) of the above theorem, it follows  that the definitions of liftable and $p$-liftable here and in loc. cit. are equivalent. Thus the theorem above sharpens Propositions 1.3 and 1.4 of \cite{KL},  which left open  whether the quaternion group  $Q_8$ is liftable: we show in  Proposition  \ref{q8}  below that $Q_8$ is not 2-liftable.

\section{Preliminaries}

Let $R$ be a commutative local ring with maximal ideal $M$ satisfying $M^2=0$ and residue field $k$ of characteristic $p>0$.
We have an exact sequence
\begin{equation}\label{modp2}
1\to A_n(R)\to \GL_n(R) \to \GL_n(k)\to 1,
\end{equation}
where $A_n(R)$ is an abelian group of exponent $p$ with a natural structure of a $\GL_n(k)$-module.

If $R=\Z/p^2\Z$ for a prime $p$, the group $A_n(R)$ is identified with the (additive) group $M_n(\F_p)$ of all $n\times n$
matrices over $\F_p$ via $1+px\mapsto \bar x$, the residue of $x$ modulo $p$. The group $\GL_n(\F_p)$ acts on $M_n(\F_p)$ by conjugation.

\begin{lemma}\label{sum}
Let $H$ be a finite group and $X$ and $Y$ two finite dimensional $k[H]$-modules. Then $X\oplus Y$ is $p$-liftable
if and only if $X$ and $Y$ are $p$-liftable.
\end{lemma}

\begin{proof}
Let $n=\dim(X)$, $m=\dim(Y)$, $R=\Z/p^2\Z$ and $k=\F_p$. Consider the diagram
\[
\xymatrix{
0  \ar[r] & M_{n+m}(k)  \ar@{=}[d]  \ar[r] &  \GL_{n+m}(R)  \ar[r]^\gamma & \GL_{n+m}(k)  \ar[r] & 1 \\
0  \ar[r] & M_{n+m}(k)  \ar[d]_\beta  \ar[r] &  E  \ar@{^{(}->}[u] \ar[d]_\delta \ar[r] & \GL_{n}(k)\times \GL_{m}(k)  \ar@{^{(}->}[u]^\alpha \ar@{=}[d] \ar[r] & 1 \\
0  \ar[r] & M_{n}(k)\oplus M_{m}(k) \ar[r]    &  \GL_{n}(R)\times \GL_{m}(R) \ar[r] &  \GL_{n}(k)\times \GL_{m}(k) \ar[r] & 1,
}
\]
where
\[
\alpha(S,T)=\begin{pmatrix}
    S & 0 \\
    0 & T \\
\end{pmatrix},\quad \beta\begin{pmatrix}
    A & B \\
    C & D \\
\end{pmatrix}=(A,D),
\]
$E$ is the inverse image under $\gamma$ of the image of $\alpha$. It consists of all invertible matrices of the form
$\begin{pmatrix}
    A & pB \\
    pC & D \\
\end{pmatrix}$ such that $A,B,C,D$ are matrices over $R$. The map $\delta$ is a homomorphism taking
$\begin{pmatrix}
    A & pB \\
    pC & D \\
\end{pmatrix}$ to $(A,D)$.

Let $\varepsilon: H\to \GL_{n}(k)\times \GL_{m}(k)$ be a homomorphism afforded by $X$ and $Y$. If  $X\oplus Y$ is $p$-liftable,
then there is a homomorphism $\eta: H\to \GL_{n+m}(R)$ such that $\gamma\eta=\alpha\varepsilon$. It follows that
the image of $\eta$ is contained in $E$. Composing $\eta$ with $\delta$, we get a lifting $H\to \GL_{n}(R)\times \GL_{m}(R)$
of  $\varepsilon$, hence $X$ and $Y$ are $p$-liftable.
\end{proof}

\begin{proposition}\label{red}\cite[Lemma 2.3]{KL}
Let $G$ be a finite group and $H\subset G$ a subgroup. If $G$ is $p$-liftable, then so is $H$.
\end{proposition}

\begin{proof}

Let $X$ be an $k[H]$-module. Consider the induced $k[G]$-module $Z:=k[G]\otimes_H X$. As a $k[H]$-module, $Z$ contains
$X$ as a direct summand: $Z=X\oplus Y$ for a $k[H]$-module $Y$. Since $Z$ is $p$-liftable as a $k[G]$-module, it is $p$-liftable
as a $k[H]$-module. By Lemma \ref{sum}, the $k[H]$-module $X$ is $p$-liftable, hence $H$ is $p$-liftable.
\end{proof}

%\begin{equation}\label{modp3}
%0\to M_n(\F_p)\to \GL_n(\Z/p^2\Z) \to \GL_n(\F_p)\to 1.
%\end{equation}

\begin{proposition}\label{sylow}\cite[Lemma 2.2]{KL}
If every Sylow subgroup of a finite group $G$ is liftable, then $G$ is liftable.
%A finite group $G$ is $p$-liftable if and only if a Sylow $p$-subgroup
%of $G$ is $p$-liftable.
\end{proposition}
\begin{proof}

%If $G$ is $p$-liftable, then a Sylow $p$-subgroup $G_p$ is $p$-liftable by Proposition \ref{red}.
%Conversely,
%suppose that $G_p$ is $p$-liftable and
Let $R$ be a commutative local ring with residue field $k$ of characteristic $p>0$ as above.
Let $f:G\to \GL_n(k)$ be a group homomorphism. The restriction $h$ of $f$
to $G_p$ lifts to a homomorphism $G_p\to \GL_n(R)$, that is the pullback of the sequence (\ref{modp2})
with respect to $h$ is splits. Equivalently, the image of the class of (\ref{modp2}) under the pullback map
$h^*:H^2(\GL_n(k), A_n(R))\to H^2(G_p, A_n(R))$ is zero.
Since $A_n(R)$ is $p$-torsion and $[G:G_p]$ is prime to $p$,
the restriction homomorphism $H^2(G, A_n(R))\to H^2(G_p, A_n(R))$ is injective. Hence
the image of the class of (\ref{modp2}) under the pullback map
$f^*:H^2(\GL_n(k), A_n(R))\to H^2(G, A_n(R))$ is also zero, that is $f$ lifts to a homomorphism
$G\to \GL_n(R)$.
\end{proof}

\section{Cyclic groups}

Set   $R=\Z/p^2\Z$ and $k=\F_p$ for a prime $p$. Let $G$ be a finite group and let
$f$ and $h$ be two elements in $k[G]$ such that $fh=0$. Choose their  lifts
$\hat f$ and $\hat h$ in $R[G]$. Then $\hat f\hat h=pu$ for some $u\in R[G]$
that is unique modulo $pR[G]$. If $\hat f+pv$ and $\hat h+pw$ are two other lifts, then
\[
(\hat f+pv)(\hat h+pw)=p(u+\hat fw+v\hat h).
\]
Thus, the residue of $u$ is uniquely determined in the quotient group $k[G]/(f k[G]+k[G] h)$.
We denote this class by $\theta(f,h)$.

\begin{proposition}\label{more}
Let $G$ be a $p$-liftable group. Then for every pair $f$ and $h$ of elements in $k[G]$ such that $fh=0$
we have $\theta(f,h)=0$.
\end{proposition}

\begin{proof}
By assumption, the left $k[G]$-module $X:=k[G]/I$, where $I=k[G]h$, admits an
$R[G]$-lifting $\widehat X$.
The surjective composition $R[G]\to k[G]\to X$ lifts to a homomorphism of left $R[G]$-modules
$R[G]\to \widehat X$ that is surjective by Nakayama. Let $\widehat I$ be its kernel.

We claim that the left ideal $\widehat I\subset R[G]$ is a lifting of $I$.
Indeed, since $\widehat X$ is free as an $R$-module, the exact sequence
$0\to \widehat I\to R[G]\to \widehat X\to 0$ is split as a sequence of $R$-modules.
It follows that $\widehat I$ is free as an $R$-module and $I\simeq \widehat I/p\widehat I$.
This proves the claim.

Let $\hat h\in \widehat I$
be a lift of $h$. By Nakayama, $\hat h$ generates $\widehat I$, that is  $\widehat I=R[G]\hat h$.

Now the map $k[G]\xra{r(h)} I$ of right multiplication by $h$
has a lifting $R[G]\xra{r(\hat h)} \widehat I$. Therefore, the kernel
$\widehat J$ of $r(\hat h)$ is a lifting of the kernel $J$ of $r(h)$. In particular, the map $\widehat J\to J$
is surjective and we can choose a lift $\hat f\in \widehat J$ of $f$. Since $\hat f\hat h=0$, we have $\theta(f,h)=0$.
\end{proof}

\begin{corollary}\label{cycle}\cite[Proposition 2.8]{KL}
Let $p>2$ be a prime and $n$ a positive integer such that $n\geq 2$ if $p=3$. Then a cyclic group of order $p^n$
is not $p$-liftable.
\end{corollary}

\begin{proof}
Let $\sigma$ be a generator of a cyclic group $G$ order $p^n$. Let $m:=p^{n-1}+1$. Consider the elements $f=(1-\sigma)^m$
and $h=(1-\sigma)^{p^n-m}$ in $k[G]$. Since $(1-\sigma)^{p^n}=0$, we have $fh=0$. We will prove that  $\theta(f,h)\neq 0$
and hence $G$ is not $p$-liftable.

Write $k[G]$ as the quotient of the polynomial ring $k[t]$ by the ideal generated by $(1-t)^{p^n}$. Since $m\leq p^n-m$ by
assumption, the class $\theta(f,h)$ is represented by $Q(\sigma)$ in
\[
k[G]/(1-\sigma)^m=k[t]/(1-t)^m,
\]
where $Q$ is the integer polynomial $[(1-t)^{p^n} -(1-t^{p^n})]/{p}$.
Setting $s:=1-t$ we have $Q=[(1-s)^{p^n} -(1-s^{p^n})]/{p}$.
The $s^{m-1}$-coefficient of $Q$ is equal to $\binom{p^n}{p^{n-1}}/p$ and hence is not divisible by $p$.
It follows that the class of $Q(\sigma)$ in $k[G]/(1-\sigma)^m=k[s]/(s^m)$ is nontrivial, hence $\theta(f,h)\neq 0$.
\end{proof}

\begin{proposition}\label{cyclic}
Let $G$ be a cyclic group of order $p^n$. Suppose that for every $i=1,2,\ldots, p^n$ there is a monic divisor $P_i$ of the
polynomial $t^{p^n}-1$ in $\Z[t]$ such that $P_i\equiv (t-1)^i$ modulo $p$. Then $G$ is liftable.
\end{proposition}

\begin{proof}
Let $R$ be a commutative local ring with maximal ideal $M$ satisfying $M^2=0$ and residue field $k$ of positive characteristic $p$.
 Choose a generator $\sigma$ of $G$.
Every finite dimensional $k[G]$-module is a direct sum of modules of the form $M_i=k[G]/(\sigma-1)^i$ for
$i=1,2,\ldots, p^n$, so it suffices to show that every $M_i$ lifts to an $R[G]$-module. Indeed,  $M_i$ lifts to the $R[G]$-module
$R[G]/(P_i(\sigma))$.
\end{proof}

\begin{corollary}\label{cycle2}\cite[Propositions 2.6 and 2.7]{KL}
Every cyclic group of order $3$ and $2^n$ is liftable.
\end{corollary}
\begin{proof}
If $p=3$, the polynomial $(t-1)^2$ lifts to the divisor $t^2+t+1$ of $t^3-1$ in $\Z[t]$.
In the case $p=2$, let $i=1,2,\ldots, 2^n$. Write $i$ is base $2$:
$i=2^{s_0}+2^{s_1}+\cdots +2^{s_m}$, where $0\leq s_0<s_1<\cdots <s_m$. The polynomial $t^{2^n}-1$ is the product
$(t-1)Q_0 Q_1\cdots Q_{n-1}$ in $\Z[t]$, where $Q_j=t^{2^j}+1$. Therefore, we can take
$P_i=Q_{s_0}Q_{s_1}\cdots Q_{s_m}$ for the lift of $(t-1)^i$ modulo $2$ in $\Z[t]$.
\end{proof}

\section{$C_2 \times C_2$ and $Q_8$ are not $2$-liftable}

Set $S=M_2(\F_2)$,  $T=M_2(\Z/4\Z)$,  and
\[
x=\begin{bmatrix}
  0 &   1\\
 1  & 1\\
 \end{bmatrix}\in S, \quad
y=\begin{bmatrix}
  0 &   1\\
 1  & 1\\
 \end{bmatrix}\in T.
 \]

\begin{proposition}\label{v4}
The Klein group $C_2 \times C_2$ is not $2$-liftable.
\end{proposition}

\begin{proof}
 Consider the following $4$-dimensional representation of $C_2 \times C_2=\langle \sigma,\tau\rangle$:
 \[
C_2 \times C_2\to \GL_2(S)=\GL_4(\F_2),\quad \sigma\mapsto
 \begin{bmatrix}
  1_S &   1_S\\
 0_S  & 1_S\\
 \end{bmatrix}, \quad
 \tau\mapsto \begin{bmatrix}
  1_S &   x\\
 0_S  & 1_S\\
 \end{bmatrix}.
 \]
 We claim that this representation is not liftable modulo $4$. Suppose it is liftable:
 \[
\sigma\mapsto A\cdot
 \begin{bmatrix}
  1_{T} &   1_{T}\\
 0_{T}  & 1_{T}\\
 \end{bmatrix}, \quad
 \tau\mapsto B \cdot \begin{bmatrix}
  1_{T} &   {y}\\
 0_{T}  & 1_{T}\\
 \end{bmatrix},
 \]
 where $A$ and $B$ are in the kernel of $\GL_2(T)\to \GL_2(S)$.
 We identify this kernel  with $M_2(S)$
 (written additively). The group $\GL_2(S)$ acts on $M_2(S)$ by conjugation.
 Let $a$ and $b$ be two matrices in $M_2(S)$ corresponding to $A$ and $B$, respectively, under this identification.
 The relations $\sigma^2=1=\tau^2$ and $\sigma\tau=\tau\sigma$ yield
 \[
 a+\sigma(a)=\begin{bmatrix}
  0 &   1\\
 0  & 0\\
 \end{bmatrix},\quad
 b+\tau(b)=\begin{bmatrix}
  0 &   x\\
 0  & 0\\
 \end{bmatrix},\quad a+ \tau(a)=b+ \sigma(b).
 \]
 in $M_2(S)$. Writing $a=(a_{ij})$ and $b=(b_{ij})$ with $a_{ij}$ and  $b_{ij}$ in $S$,
 we have the following relations in $S$:
 \[
 a_{21}=b_{21}=0,
 \]
 \[
 a_{11}+a_{22}=1,
 \]
 \[
 b_{11}x + xb_{22}=x,
 \]
\[
 a_{11}x + xa_{22}=b_{11} + b_{22}.
 \]
Consider the subspace
\[
V:=[x, S]=\{xs+ sx,\ s\in S\}\subset S.
\]
Note that $xV=V=x\inv V$ since $x(xs+ sx)=x(xs)+(xs)x$. We then have the following equalities in $S/V$:
\[
1=x\inv (b_{11}x + xb_{22})=b_{11} + b_{22}=a_{11}x + xa_{22}=
x(a_{11}+a_{22})=x,
\]
hence $x+1\in V$, a contradiction since $\operatorname{Trace}(x+1)\neq 0$.
\end{proof}

\begin{proposition}\label{q8}
The quaternion group $Q_8$ is not $2$-liftable.
\end{proposition}

\begin{proof}

Note the relation $x^2+x+1=0$ in $S$.

Consider the quaternion group $Q_8=\langle \sigma, \tau\ |\ \sigma^2=\tau^2, \sigma^4=1, \tau\sigma=\sigma^3 \tau \rangle$ and its $6$-dimensional
representation over $\F_2$ given by
\[
 Q_8\to \GL_3(S)=\GL_6(\F_2),\quad \sigma\mapsto
j=\begin{bmatrix}
  0 &  0 & 1\\
 1 &  0 & 1\\
  0 &  1 & 1\\
\end{bmatrix}, \quad
 \tau\mapsto k=\begin{bmatrix}
0 &  x  & 1\\
 x &   {x}^2 &  x\\
{x}^2&   0 &  x \\
\end{bmatrix}.
 \]
 We claim that this representation is not liftable modulo $4$. Suppose it is liftable:
\[
\sigma\mapsto J+2  A,
\quad \tau\mapsto K+2 B,
\]
where $A$ and $B$ are two $3\times 3$ matrices over $T$ and
\[
J=\begin{bmatrix}
  0 &  0 & 1\\
 1 &  0 & 1\\
  0 &  1 & 1\\
\end{bmatrix},
\quad K=\begin{bmatrix}
0 &  y  & 1\\
  y &   {y}^2 &  y\\
{y}^2&   0 & y \\
\end{bmatrix}
\]
in $M_3(T)$.
The equality $\sigma^2=\tau^2$ yields $(J+2  A)^2= (K+2 B)^2$ in $M_3(T)$. Since
 \[
{J}^2=\begin{bmatrix}
  0 &  1 & 1\\
 0 &  1 & 2\\
  1 &  1 & 2\\
\end{bmatrix},
\quad {K}^2=\begin{bmatrix}
2{y}^2 &  1  & 1\\
  2 &   1 &  0\\
1 &  1 & 2{y}^2 \\
\end{bmatrix},
\]
we get an equality
\[
2(J A+A J)=\begin{bmatrix}
2{y}^2 &  0  & 0\\
2 &   0 &  2\\
0 &   0 & 2{y} \\
\end{bmatrix}+2(K B+ B K)
\]
in $M_3(T)$, hence
\begin{equation}\label{comm}
ja+ aj=\begin{bmatrix}
{x}^2 &  0  & 0\\
1 &   0 &  1\\
0 &   0 & x \\
\end{bmatrix}+( kb+ bk)
\end{equation}
in $M_3(S)$, where $a$ and $b=(b_{ij})$ in $M_3(S)$ are the residues modulo $2$ of $A$ and $B$, respectively.
A computation shows that the trace
in $S$ of the $3\times 3$ matrix $ja+ aj$ is zero and the trace of $kb+ bk$ is equal to
\[
[x, b_{12}]+[x, b_{21}]+[x, b_{32}]+[x, b_{33}]+[x^2, b_{13}]+[x^2, b_{22}].
\]
This is contained in $V=[x,S]$ since $[x^2,c]=[x+1,c]=[x,c]\in V$. It follows from (\ref{comm})
that $1=x^2+x\in V$, hence $x\in xV=V$, a contradiction since $\operatorname{Trace}(x)\neq 0$.
\end{proof}

\section{$C_3 \times C_3$ is not $3$-liftable}

\begin{proposition}\label{33}\cite[Claim 5.4]{MS26}
The group $C_3 \times C_3$ is not $3$-liftable.
\end{proposition}

\begin{proof}
For every $1 \leq i < j \leq 3$, let $e_{ij}$ be the matrix with
$1$ on the $(i j)$-entry, and $0$ everywhere else and set $\sigma_{ij}:=1+e_{ij}$.
Let $C_3 \times C_3\to \GL_3(\F_3)$ be a $3$-dimensional representation taking generators to $\sigma_{12}$ and $\sigma_{13}$,
respectively. If this homomorphism lifts to $\GL_3(\Z/9\Z)$,  there should exist two elements $\rho,\tau\in {\GL}_3(\Z/9\Z)$ such that $\rho^3=\tau^3=[\rho,\tau]=1$, $\rho$ reduces to $\sigma_{12}$ modulo $3$, and $\tau$ reduces to $\sigma_{13}$ modulo $3$. We now prove that such $\rho$ and $\tau$ do not exist. We have $\rho=A{\sigma}_{12}$ and
$\tau=B{\sigma}_{13}$ in $\GL_3(\Z/9\Z)$,
for some $A, B\in A_3(\Z/9\Z)$. Identifying  $A_3(\Z/9\Z)$ with the additively written group $M_3(\F_3)$,
we rewrite the relations $\rho^3=\tau^3=[\rho,\tau]=1$ in the form
\[
N_{12}(a) = -e_{12},\quad N_{13}(b) = -e_{13}\quad\text{and}\quad  (\sigma_{13}-1)a = (\sigma_{12} - 1)b
\]
in $M_3(\F_3)$, where $a,b\in M_3(\F_3)$ and
$N_{ij}:= 1+\sigma_{ij}+\sigma_{ij}^2$.
Letting $a=(a_{ij})$, a matrix computation shows
\[
N_{12}(a)=
\begin{pmatrix}
    0 & a_{21} & 0 \\
    0 & 0 & 0 \\
    0 & 0 & 0
\end{pmatrix},\qquad
-(\sigma_{13}-1)a+(\sigma_{12}-1)b=\begin{pmatrix}
    * & * & * \\
    * & * & a_{21} \\
    * & * & *
\end{pmatrix}.
\]
Thus $-1=a_{21}=0$, a contradiction.
\end{proof}

\section{Proof of Theorem \ref{main}}

\begin{proof}
$(1)\Rightarrow (2)$ is trivial.

$(2)\Rightarrow (3)$: Let $G$ be a $p$-liftable finite group for every prime $p$.
By Proposition \ref{red}, a Sylow $p$-subgroup $G_p$ of $G$ is $p$-liftable for all $p$.
If $p>3$ then $G_p=1$ since otherwise $G_p$ would contain a cyclic group of order $p$ which is not
$p$-liftable by Corollary \ref{cycle}. This contradicts Proposition \ref{red}.

The subgroup $G_3$ has order at most $3$. Indeed, if $|G_3|>3$, then $G_3$ contains a subgroup of order $9$, hence
$G_3$ contains either $C_3\times C_3$ or $C_9$. The latter two groups are not $3$-liftable by
Proposition \ref{33} and Corollary \ref{cycle}. Again, we get a contradiction by Proposition \ref{red}.

We claim that every abelian subgroup $H\subset G_2$ (which is $2$-liftable by Proposition \ref{red}) is cyclic.
Indeed, if $H$ were not cyclic, it would contain a subgroup $C_2\times C_2$. This contradicts Propositions
\ref{v4} and \ref{red}. The claim is proved.

By \cite[Chapter IV, Corollary 6.6]{AM}, a $2$-group with all abelian subgroups cyclic is either cyclic or
generalized quaternion group $Q_{2^n}$, $n\geq 3$,
generated by $\sigma$ and $\tau$ subject to the relations $\sigma^{2^{n-2}}=\tau^2$, $\tau^4=1$ and
$\tau\sigma\tau\inv =\sigma\inv$. The subgroup of $Q_{2^n}$ generated by $\sigma^{2^{n-3}}$ and $\tau$ is isomorphic
to $Q_8$ which is not $2$-liftable by Proposition $\ref{q8}$. It follows from Proposition \ref{red} that $G_2$ is cyclic.
Thus, we proved that all Sylow subgroups of $G$ are cyclic and the order of $G$ divides $3\cdot 2^n$.
It follows from the proof of \cite[Proposition 2.4]{KL} that $G$ is one of the groups in the list $(3)$.

$(3)\Rightarrow (1)$: By Proposition \ref{sylow} it suffices to prove that all Sylow subgroups $G_p$ of $G$ are liftable.
The group $G_p$ is either trivial or cyclic of order $3$ or $2^n$. By Corollary \ref{cycle2}, $G_p$ is liftable.
\end{proof}

\def\cprime{$'$} \def\cprime{$'$}

%\bibliography{bibfile}
%\bibliographystyle{acm}

\end{document}